\documentclass{article}
\usepackage{amssymb, amsmath, amsthm}

\usepackage[english]{babel} 
\usepackage[utf8]{inputenc} 
\usepackage{textcomp}
\usepackage{caption}
\usepackage{pdfpages}
\usepackage{chngpage}
\usepackage{hyperref}
\usepackage{comment}
\usepackage{graphicx}
\usepackage{tikz}
\usepackage{soul,color}
\usetikzlibrary{arrows}
\usepackage{todonotes}
\usepackage{comment}
\usepackage{listings}
\usepackage{framed}

\title{Combinatorial classification of quantum lens spaces}
  \author{\textbf{Peter Lunding Jensen} \\
          {\small Department of Mathematical Sciences, University of Copenhagen, Denmark}\\ 
          \href{mailto:dvf408@alumni.ku.dk}{dvf408@alumni.ku.dk} \vspace{10pt} \\
          \textbf{Frederik Ravn Klausen} \\ 
         {\small Department of Mathematical Sciences, University of Copenhagen, Denmark}\\
          \href{mailto:tlk870@alumni.ku.dk}{tlk870@alumni.ku.dk} \vspace{10pt}  \\   
          \textbf{Peter M.\ R.\ Rasmussen}\footnote{Currently at Computer Science Department, University of California Los Angeles.}  \\
          {\small Department of Mathematical Sciences, University of Copenhagen, Denmark}\\ 
          \href{mailto:rasmussen@cs.ucla.edu}{rasmussen@cs.ucla.edu}}

\newcommand\numberthis{\addtocounter{equation}{1}\tag{\theequation}}

\newcommand{\ceil}[1] {
\left\lceil #1 \right\rceil
}

\newcommand{\N}{\mathbb{N}}
\newcommand{\Z}{\mathbb{Z}}

\newcommand{\KKK}{\mathbb K}
\newcommand{\qlsm}[3]{C(L_q(#1;({#3}'_1,\dots, {#3}'_{#2})))}
\newcommand{\qls}[3]{C(L_q(#1;({#3}_1,\dots, {#3}_{#2})))}

\newcommand{\A}[2]{\mathsf{A}_{(#1;\overline {#2})}}
\newcommand{\Am}[2]{\mathsf{A}_{(#1;\overline {#2}')}}

\newcommand{\B}[2]{\mathsf{B}_{(#1;\overline {#2}
)}}
\newcommand{\Bm}[2]{\mathsf{B}_{(#1;\overline {#2}'
)}}

\newcommand{\AAA}[3]{\mathsf{A}_{(#1;({#3}_1,\dots, {#3}_{#2}))}}
\newcommand{\AAAm}[3]{\mathsf{A}_{(#1;({#3}'_1,\dots, {#3}'_{#2}))}}
\newcommand{\BBB}[2]{\mathsf{B}_{(#1;#2)}}

\newcommand{\NN}[2]{\mathsf{N}_{(#1;\overline {#2})}}
\newcommand{\MM}[2]{\mathsf{M}_{(#1;\overline {#2})}}
\newcommand{\NNN}[2]{\mathsf{N}_{(#1;(#2))}}
\newcommand{\MMM}[2]{\mathsf{M}_{(#1;(#2))}}

\newcommand{\ind}[2]{\langle #1, #2 \rangle}

\newcommand{\abs}[1]{\left\vert #1\right\vert}

\newcounter{thmcounter}
\numberwithin{thmcounter}{section}

\newtheorem{theorem}[thmcounter]{Theorem}
\newtheorem{definition}[thmcounter]{Definition}
\newtheorem{lemma}[thmcounter]{Lemma}
\newtheorem{corollary}[thmcounter]{Corollary}

\newtheorem{notation}[thmcounter]{Notation}
\newtheorem{conjecture}[thmcounter]{Conjecture}

\begin{document}
\maketitle
\begin{abstract}
We answer the question of how large the dimension of a quantum lens space must be, compared to the primary parameter $r$, for the isomorphism class to depend on the secondary parameters. Since classification results in C*-algebra theory reduces this question to one concerning a certain kind of $SL$-equivalence of integer matrices of a special form, our approach is entirely combinatorial and based on the counting of certain paths in the graphs shown by Hong and Szyma\'nski to describe the quantum lens spaces.
\end{abstract}
\newpage

\section{Introduction}
In a seminal paper by Hong and Szyma\'nski \cite{jhhws:qlsga} an important class of \emph{quantum lens spaces} $\qls rnm)$  was given a description as $C^*$-algebras arising from certain graphs -- or their adjacency matrices -- in the vein of Cuntz and Krieger \cite{jcwk:cctmc}. These graphs can be read off directly from the data $(r;(m_1,\dots,m_n))$ determining the quantum lens space, where $r>2$ are integers and $m_i$ are units of $\Z/r\Z$. Using this characterisation, it is easy to see that $\qls{r}{n}{m}$ can only be isomorphic to  $\qls{r'}{n'}{m'}$ when $r=r'$ and $n=n'$, and this raises the important question of to what extent the choice of the units can influence the $C^*$-algebras. 

To answer such questions, one appeals naturally to the classification theory for $C^*$-algebras by $K$-theory, as indeed a large class of Cuntz-Krieger algebras were classified by Restorff in \cite{gr:cckasi}. Unfortunately, the quantum lens spaces fall outside this class, and indeed, outside any class considered at the time \cite{jhhws:qlsga} was written. 
Thus, apart from noting that the $m_i$ can obviously not influence the $C^*$-algebras when $n\leq 3$,  Hong and  Szyma\'nski left the question open.

Quantum lens spaces are still a subject of interest, however, see for instance Arici, Brain, and Landi \cite{fasbgl:qlsnyto} and Brzezi{\'n}ski and Szyma{\'n}ski \cite{tbws:qlsnyet}, and using recent classification results obtained for Cuntz-Krieger algebras with uncountably many ideals, Eilers, Restorff, Ruiz, and S\o rensen in \cite{seerapws:gccfg} managed to reduce this question to elementary matrix algebra and to prove that when $n=4$ there are precisely two different $\qls{r}{n}{m}$ when $r$ is a multiple of 3, and only one when $r$ is not.

Søren Eilers made computer experiments for other $r$ and $n$ which suggested that the quantum lens spaces are unique when $n< s$ for $s$ the smallest even number strictly larger than the smallest divisor of $r$ which is not $2$, and that at least two choices of $m_i$ give different $C^*$-algebras when $n\geq s$. It is the aim of the paper at hand to provide the combinatorial insight needed to prove that this in fact is the case, and to study the number of different $C^*$-algebras that can be obtained by varying the $m_i$.

We will not work directly on questions of isomorphism of the $C^*$-algebras, and hence, \emph{no prior knowledge on $C^*$-algebras or their classification theory is required}. Instead we study the equivalent notion of $SL$ equivalence of the graphs associated to the given data. Indeed, a result of \cite{seerapws:gccfg} states that the following are equivalent  
\begin{itemize}
\item $\qls{r}{n}{m}\otimes \KKK\simeq \qls{r}{n}{m'}\otimes \KKK$
\item There exist integer matrices $U,V$ both of the form
\[
\begin{bmatrix}
1&*&*&\dots&*\\
 &1&*&&\\
&&\ddots&\ddots&\vdots\\
&&&1&*\\
&&&&1
\end{bmatrix}
\]
so that $U(\AAA{r}{n}{m}-I)=(\AAA{r}{n}{m'}-I)V$
\end{itemize}
The exact notation and definitions will be given in Section 2 together with the rudimentary results needed for our classification. Section 3 handles the most general case, basically establishing the influence of the odd prime divisors of the parameter $r$ on the number of $C^*$-algebras emerging by varying the $m_i$. A lower bound on the number of such $C^*$-algebras is found and for $4\nmid r$ the exact $s$ such that the $C^*$-algebra is unique for $n<s$ is determined. The special case of finding $s$ when $4\mid r$ is then dealt with in Section 4.

The main result of the paper is Theorem \ref{theorem_5_1} which combines the results of Section 3 and 4 to find for every $r>2$ the $s$ such that the $C^*$-algebra is unique for every $n<s$. The other major achievement is Theorem \ref{classesInequality} which bounds the number of different quantum lens spaces arising for some $r>2$ and $n\in \N$. Based on computer experiments, we conjecture that this bound is in fact an equality when $4\nmid r$ (Conjecture \ref{classes_conjecture}).

\section{Preliminaries}
Initially, we dedicate a section to setting the stage. We establish notation, definitions, and find initial results that will assist in showing the later sections' classification results.

\subsection{Number theoretical notation}
\begin{definition}
We let $Z_n$ denote the multiplicative group of integers modulo $n$. That is $Z_n = (\Z/n\Z)^*$.
\end{definition}
\begin{notation}
We write $p^k\mid\mid n$ if $p^k\mid n$ and $p^{k+1}\nmid n$, i.e. $k$\ is the greatest power of $p$ dividing $n$.
\end{notation}

\begin{notation}
To ease notation we write the reduction of an integer $a$ calculated modulo $r$ as $ [a]_r $, i.e. we always have $0 \leq [a]_r \leq r-1.  $
\end{notation}

\subsection{The graph}
This section will introduce a definition of the graph $\MMM{r}{m_1, \dots, m_n}$, arising from the quantum lens space $C(L_q(r;(m_1, \dots, m_n)))$ as defined in \cite{jhhws:qlsga}. Further, we introduce another graph $\NNN{r}{m_1, \dots, m_n}$, which is easier to work with in the combinatorial setting, but has similar properties in a sense that will be made clear. 
\begin{definition} \label{def1}
Let $r>2$ and $\overline m = (m_1, \dots, m_n) \in (Z_r)^n$ for some $n\in \N$. Then we define a directed graph $\MM{r}{m}$ in the following way:
\begin{itemize}
\item For every pair $s, t$ with $1\leq s \leq n$ and $0\leq t<r$ there is a vertex $g_{s, t}$.
\item There is a directed edge from $g_{s_1, t_1}$ to $g_{s_2, t_2}$ if and only if $s_1\leq s_2$ and $t_2=[t_1+m_{s_1}]_r $.
\end{itemize}
For every $s\in \N$ we will call the subgraph consisting of the vertices $\{g_{s, x}\mid 0\leq x<r\}$ the $s$th subgraph of $\MM{r}{m}$, and we will call a vertex of the form $g_{s, c}$ a $c$-vertex.
\end{definition}
An example of the graph $\MMM{5}{1, 2, 1}$ is sketched in Figure \ref{graf1}.
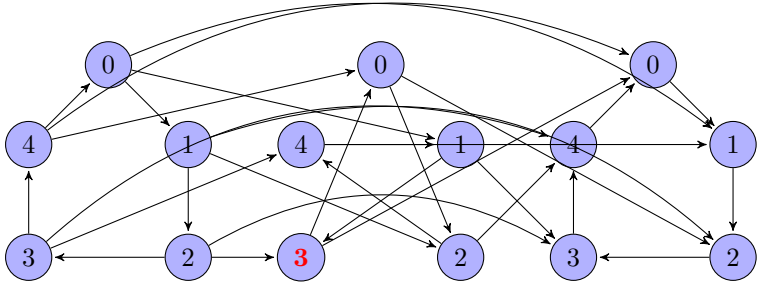
\begin{figure}[!h]
\begin{center}
\begin{tikzpicture}[->,>=stealth',shorten >= 1pt, auto,node distance=1.5 cm, main node/.style={circle,fill=blue!30, draw}]

  \node[main node] (0)                      {0};
  \node[main node] (1) [below right of=0]   {1};
  \node[main node] (2) [below of=1]         {2};
  \node[main node] (4) [below left of=0]    {4};  
  \node[main node] (3) [below of=4]         {3};
  
  \node[main node] (5) [right of=1]         {4};
  \node[main node] (9) [above right of=5]   {0};
  \node[main node](6) [below of=5]         {\color{red} \textbf{3}};
  \node[main node] (7) [below right of=9]   {1};
  \node[main node] (8) [below of=7]         {2};

  \node[main node] (10) [right of=7]         {4};
  \node[main node] (11) [above right of=10]   {0};
  \node[main node] (12) [below of=10]         {3};
  \node[main node] (13) [below right of=11]   {1};
  \node[main node] (14) [below of=13]         {2};

\path
(0) edge node[left] {} (1)
(1) edge node[left] {} (2)
(2) edge node[left] {} (3)
(3) edge node[left] {} (4)
(4) edge node[left] {} (0)

(5) edge node[left] {} (7)
(6) edge node[left] {} (9)
(7) edge node[left] {} (6)
(8) edge node[left] {} (5)
(9) edge node[left] {} (8)

(10) edge node[left] {} (11)
(11) edge node[left] {} (13)
(12) edge node[left] {} (10)
(13) edge node[left] {} (14)
(14) edge node[left] {} (12)

(0) edge [bend left] node[left] {} (13)
(0) edge [left] node[left] {} (7)
(1) edge [left] node[left] {} (8)
(1) edge [bend left] node[left] {} (14)
(2) edge [left] node[left] {} (6)
(2) edge [bend left] node[left] {} (12)
(3) edge [left] node[left] {} (5)
(3) edge [bend left] node[left] {} (10)
(4) edge [left] node[left] {} (9)
(4) edge [bend left] node[left] {} (11)

(6) edge [left] node[left] {} (11)
(5) edge [left] node[left] {} (13)
(7) edge [left] node[left] {} (12)
(9) edge [left] node[left] {} (14)
(8) edge [left] node[left] {} (10)
;
    
\end{tikzpicture}
\end{center}
\caption{Example of $\MM{r}{m}$ with $n=3$, $r=5$ and $m=(1,2,1)$.
The red 3 denotes the vertex $g_{2,3}$.}\label{graf1}
\end{figure}

\begin{definition}\label{def2}
Let $r>2$ and $\overline m = (m_1, \dots, m_n) \in (Z_r)^n$ for some $n\in \N$. Then we define a directed graph $\NN{r}{m}$ in the following way:
\begin{itemize}
\item For every pair $s, t$ with $1\leq s \leq n$ and $0\leq t<r$ there is a vertex $c_{s, t}$.
\item There is a directed edge from $c_{s_1, t_1}$ to $c_{s_2, t_2}$ in the following two cases 
\begin{itemize}
\item $s_1+1= s_2$ and $t_2=t_1$
\item $s_1=s_2$ and $t_2 = [t_1+ m_{s_1}]_r$.
\end{itemize}
\end{itemize}
For every $s$ we will call the subgraph consisting of the vertices $\{c_{s, x}\mid 0\leq x<r\}$ the $s$th subgraph of $\NN{r}{m}$, and we will call a vertex of the form $c_{s, t}$ a $t$-vertex.
\end{definition}
\noindent
Here is the graph we would rather look at. Instead of having edges from a subgraph to all the subgraphs after it, it only has edges to the one just after it. This edge will always go from $c_{s, t}$ to $c_{s+1, t}$. We show an example of the graph on Figure \ref{graf2}.
\begin{figure}[!h]
\begin{center}
\begin{tikzpicture}[->,>=stealth',shorten >= 1pt, auto,node distance=1.5 cm, main node/.style={circle,fill=blue!30, draw}]

  \node[main node] (0)                      {0};
  \node[main node] (1) [below right of=0]   {1};
  \node[main node] (2) [below of=1]         {2};
  \node[main node] (4) [below left of=0]    {4};  
  \node[main node] (3) [below of=4]         {3};
  
  \node[main node] (5) [right of=1]         {4};
  \node[main node] (9) [above right of=5]   {0};
  \node[main node] (6) [below of=5]         {3};
  \node[main node] (7) [below right of=9]   {1};
  \node[main node] (8) [below of=7]         {2};

  \node[main node] (10) [right of=7]         {4};
  \node[main node] (11) [above right of=10]   {0};
  \node[main node] (12) [below of=10]         {3};
  \node[main node] (13) [below right of=11]   {1};
  \node[main node] (14) [below of=13]         {2};

\path
(0) edge node[left] {} (1)
(1) edge node[left] {} (2)
(2) edge node[left] {} (3)
(3) edge node[left] {} (4)
(4) edge node[left] {} (0)

(5) edge node[left] {} (7)
(6) edge node[left] {} (9)
(7) edge node[left] {} (6)
(8) edge node[left] {} (5)
(9) edge node[left] {} (8)

(10) edge node[left] {} (11)
(11) edge node[left] {} (13)
(12) edge node[left] {} (10)
(13) edge node[left] {} (14)
(14) edge node[left] {} (12)

(0) edge [bend left] node[left] {} (9)
(1) edge [bend left] node[left] {} (7)
(4) edge [bend left] node[left] {} (5)
(2) edge [bend left] node[left] {} (8)
(3) edge [bend left] node[left] {} (6)

(6) edge [bend left] node[left] {} (12)
(7) edge [bend left] node[left] {} (13)
(5) edge [bend left] node[left] {} (10)
(9) edge [bend left] node[left] {} (11)
(8) edge [bend left] node[left] {} (14)

;
    
\end{tikzpicture}
\end{center}

\caption{Example of $\NN{r}{m}$ where $n=3$, $r=5$, and $m=(1,2,1)$. }\label{graf2}
\end{figure}
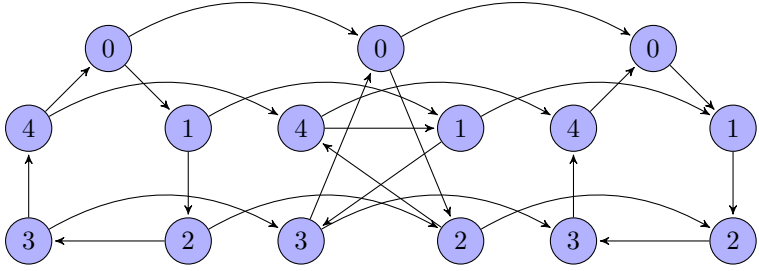

\begin{definition}\label{A}
Let $r>2$ and $\overline m = (m_1, \dots, m_n) \in (Z_r)^n$ for some $n\in \N$. Then we let $\A{r}{m}$ be the matrix satisfying that $\A{r}{m}\ind{i}{j}$ is the number of directed paths in $\MM{r}{m}$ from the $0$-vertex of the $i$th subgraph to the $0$-vertex of the $j$th subgraph that does not pass through the 0-vertex of any other subgraph. We call a path that satisfies these criteria \emph{legal}.
\end{definition}

\begin{definition}\label{B}
Let $r>2$ and $\overline m = (m_1, \dots, m_n) \in (Z_r)^n$ for some $n\in \N$. Then we let $\B{r}{m}$ be the matrix satisfying that $\B{r}{m}\langle i, j\rangle$ is the number of directed paths on $\NN{r}{m}$ from the $0$-vertex of the $i$th subgraph to the $0$-vertex of the $j$th subgraph which do not exclusively visit 0-vertices and which do not visit the 0-vertex of any other subgraph except if all the following vertices of the path are 0-vertices. We will call a path that satisfies these criteria \emph{legal}.
\end{definition}

We introduce this new graph definition $\NN{r}{m}$ because it is easier to work with than $\MM{r}{m}$. Note we always calculate indices in subgraphs modulo $r$.

\begin{lemma}\label{equalMatrices}
Let $r>2$ and $\overline m\in (Z_r)^n$ be given. Then $\A{r}{m} = \B{r}{m}$.
\end{lemma}
\begin{proof} 
There is a bijection between the edges of $\MM{r}{m}$ and paths of $\NN{r}{m}$ as follows. The edge $g_{s_1, t_1}\rightarrow g_{s_2, t_1+m_{s_1}}$ of $\MM{r}{m}$ corresponds to the path 
\begin{equation*}
	c_{s_1, t_1}\rightarrow c_{s_1, t_1 + m_{s_1}} \rightarrow c_{s_1+1, t_1 + m_{s_1}} \rightarrow \dots \rightarrow c_{s_2, t_1 + m_{s_1}}
\end{equation*}
on $\NN{r}{m}$.That this is a bijection follows immediately from the fact that the edge and path are both uniquely determined by $s_1$, $s_2$, and $t_1$. 

Now, we need to establish a bijection between the legal paths on $\MM{r}{m}$ and the legal paths on $\NN{r}{m}$. This happens naturally by translating any edge in a legal path on $\MM rm$ into a subpath of the form above of a legal path on $\NN mr$. That this map has an inverse follows easily since any legal path in $\NN{r}{m}$ consists of subpaths of the above form where a new subpath starts whenever we stay in the same subgraph. Further, we have that the constraint of Definition \ref{A} translates into the constraint of Definition \ref{B} an edge from the $t$th subgraph to the 0-vertex of the $n$th subgraph in Definition \ref{A} corresponds to going to the $0$-vertex in the $t$th subgraph and then visiting 0-vertices exclusively until reaching the 0-vertex of the $n$th subgraph in Definition \ref{B}.
\end{proof}

\subsection{Equivalence classes}
The overall aim of the article is to classify the quantum lens spaces, which is a problem that Theorem 7.1 of Section 7.2 of \cite{seerapws:gccfg} reduces to a question of $SL$ equivalence, hence elementary matrix algebra. 
\begin{theorem}[Eilers, Restorff, Ruiz, and Sørensen] 
Let $r>2$ and $\overline m, \overline m'\in (Z_r)^n$ be given. The following are equivalent:
	\begin{itemize}
\item $\qls{r}{n}{m}\otimes \KKK\simeq \qlsm{r}{n}{m}\otimes \KKK$.
\item There exist matrices $U,V$ both of the form
\[
\begin{bmatrix}
1&*&*&\dots&*\\
 &1&*&&\\
&&\ddots&\ddots&\vdots\\
&&&1&*\\
&&&&1
\end{bmatrix}
\]
so that $U(\AAA{r}{n}{m}-I)=(\AAAm{r}{n}{m}-I)V$.
\end{itemize}
\end{theorem}
\noindent
Thus, determining whether or not two quantum lens spaces, $C(L_q(r;(m_1',\dots, m_n')))$ and $\qls{r}{n}{m}$, are isomorphic comes down to whether or not the matrices $\A{r}{m}$ and $\Am{r}{m}$ (or $\B{r}{m}$ and $\Bm{r}{m}$ by Lemma \ref{equalMatrices}) are equivalent with respect to the equivalence relation, $\sim$, defined below.

\begin{definition}
We will say that two matrices $C$ and $D$ are upper triangular equivalent, written $C\cong D$ if  there exist upper triangular matrices, $X, Y$, with 1 in every entry of the diagonal such that $XC=DY$.
\end{definition}
Equivalently, the matrices $C$ and $D$ are upper triangular equivalent, if there is a series of pivots transforming $C$ into $D$ with the restrictions that
\begin{enumerate}
    \item a multiple of row $k$ can only be added to row $l$ if $k>l$
    \item a multiple of column $k$ can only be added to column $l$ if $k<l$.
\end{enumerate} 
\noindent
Note that this is clearly an equivalence relation since such upper triangular matrices are invertible.
\begin{definition}\label{maindefinition}
We say that two matrices, $A, B$ are $\sim$-equivalent, if $$A-I\cong B-I.$$
\end{definition}
\noindent
In particular, we are interested in efficiently deciding the number of equivalence classes given $n$ and $r>2$ and deciding whether or not two graphs belong to the same equivalence class.
\begin{definition}
Let $r>2$ and $n\in \N$ be given. Then we define
\begin{equation*}
  S_{r, n} = \left\{ \A{r}{m} \mid \overline m \in (Z_r)^n \right\} = \left\{ \B{r}{m} \mid \overline m \in (Z_r)^n \right\}
\end{equation*}
as the set of all matrices produced by vectors of length $n$ with parameter $r$.
\end{definition}

\begin{definition}
Let $r>2$ and $n\in \N$ be given. Then $\varphi_r(n)$ denotes the number of elements of $S_{r, n} / \sim$ and $\widetilde \varphi(r)$ denotes the least $n$ such that $\varphi_r(n)>1$.
\end{definition}
\noindent
Thus, our goal in this paper is to find a bound for $\varphi_r$ given $r$ and to express $\widetilde\varphi$ in closed form.

\subsection{Invariants}
In this section we establish some invariants and properties in relation to changes to the vector $\overline m$ in $\NN{r}{m}$.
\begin{lemma}\label{invariantFirstLast}
The matrix $\B{r}{m}$ does not depend on the choice of $m_1$ and $m_n$.
\end{lemma}

\begin{proof}
If $n=1$ this is obvious, so assume $n>1$. Consider legal paths in $\NNN{r}{m_1, \dots, m_n}$ from the 0-vertex of the first subgraph of to the 0-vertex of the $j$th subgraph for $j>1$. No matter what $m_1$ is there is exactly one way to reach any of the vertices of the second subgraph from the 0-vertex of the first subgraph. Thus, the number of such directed paths is independent of $m_1$ and the first part follows.

Now, consider the last subgraph. Once it is reached, there is exactly one way to reach the 0-vertex, once it is reached, so this does not depend on $m_n$.
\end{proof}

\begin{lemma}\label{multiply}
Let $r>2$, $\overline m \in (Z_r)^n$, and $b\in Z_r$. Then $\B{r}{m} = \BBB{r}{b\cdot \overline m}$.
\end{lemma}
\begin{proof}
We will show that there is a bijection between the legal paths of $\B{r}{m}$ and $\BBB{r}{b\cdot \overline m}$ as follows. Let $\gamma$ be a legal path
\begin{align*}
    c_{s_1, 0}=c_{s_1, t_1}\to c_{s_2, t_2}\to \dots \to c_{s_q, t_q}=c_{s_q, 0}
\end{align*}
on $\NN{r}{m}$. Our bijection sends the legal $\gamma$ to the path $\omega$ on $\NNN r{b\cdot \overline m}$  given by 
\begin{align*}
    c_{s_1, 0}=c_{s_1, [b\cdot t_1]_r}\to c_{s_2, [b\cdot t_2]_r}\to \dots \to c_{s_q, [b\cdot t_q]_r}=c_{s_q, 0}.
\end{align*}
That the map is injective follows since multiplication by $b\in Z_r$ is an injection $Z_r\to Z_r$. Further, it is easy to see that all legal paths on $\NN{r}{m}$ will be mapped to legal paths on $\NNN r{b\cdot \overline m}$ since multiplication by $b$ does not change the positions of the 0-vertices in a path. Thus, there is an injection from the legal paths on $\NN{r}{m}$ to the legal paths on $\NNN r{b\cdot \overline m}$ and by the same argument there must be an injection from the legal paths on $\NNN r{b\cdot \overline m}$ to the legal paths on $\NN{r}{m}$. It follows that said map is a bijection and we are done. 
\end{proof}

\begin{corollary}\label{change3}
Let $\overline m  = (m_1,m_2, \dots, m_{n-1},m_n) \in  (Z_r)^n$. Then there exists an $\overline m' \in (Z_r)^n$ with 1 in the first, last and $k$-th index, i.e.\
 $$\overline m' = (1,m_2', \dots m_{k-1}', 1, m_{k+1}' \dots, m_{n-1}',1),$$ such that $\B{r}{m} = \Bm{r}{m}$. 
\end{corollary}
\begin{proof}
Take $b$ to be the inverse in $Z_r$ of $m_{k}$ in Lemma \ref{multiply}. Then $\B{r}{m} = \BBB{r}{ m_{k}^{-1} \cdot \overline m} = \Bm{r}{m}$ where the last equality follows from Lemma \ref{invariantFirstLast}. 
\end{proof}

\subsection{Entry specific properties and formulae} 
To proceed with any further results we need some combinatorial formulae and properties to be in place. 
\begin{theorem}\label{closedFormula}
Let $\overline 1 = (1, \dots 1)$. Then $\B{r}{1}\langle i, j\rangle = \binom{r-1+(j-i)}{j-i}$.
\end{theorem}
\begin{proof}
Since every directed edge in $\NN{r}{1}$ either goes from $c_{s, t}$ to $c_{s+1, t}$ or from $c_{s, t}$ to $c_{s, t+1}$, we can characterise any directed path from the 0-vertex of the $i$th subgraph to the 0-vertex of the $j$th subgraph satisfying Definition \ref{B} by an $j-i+1$-tuple $(a_0, \dots, a_{j-i})$, such that $a_s$ is the number of edges of the form $c_{s,t}\rightarrow c_{s,t+1}$ that occur in the path. A necessary and sufficient condition for such a tuple to characterise a directed path of the desired form is that $\sum_{l=0}^{j-i}a_l = r$, $a_s\geq 0$ for all $s>0$, and $a_0>0$.

Thus, $\B{r}{1}\langle i, j\rangle$ is equal to the number of ways that $r$ can be written as the sum of $j-i+1$ non-negative integers, where the first one has to be at least 1. This is equivalent to the number of ways to write $r-1$ as the sum of $j-i+1$ non-negative integers. The latter being a known combinatorial problem, we get
\begin{equation*}
 \B{r}{1}\langle i, j\rangle = \binom{r-1+(j-i)}{j-i}.
\end{equation*}
\end{proof}

\begin{corollary} \label{smallcase}
Let $r>2$ and $\overline m \in (Z_r)^n$. Then $\B{r}{m}\langle i, i\rangle =1$, $\B rm = \langle i, i+1\rangle=r$, and $\B rm \langle i, i+2\rangle=\frac{r(r+1)}{2}$ for all $i$.
\end{corollary}
\begin{proof} 
When we consider only $\B{r}{m}\langle i, i\rangle, \B{r}{m}\langle i, i+2\rangle$, and $\B{r}{m}\langle i, i+2\rangle$, their values depend solely on the vector $(m_i, m_{i+1}, m_{i+2})$, so we can assume by Corollary \ref{change3} that $m_i=m_{i+1}=m_{i+2}=1$. The conclusion then follows trivially from Theorem \ref{closedFormula}
\end{proof}

From the corollary we immediately obtain the following result, which also appeared in \cite{seerapws:gccfg} and we will note for future use. 

\begin{corollary}\label{3divides}
Let $r>2$. Then $\widetilde \varphi(r) \geq 4$. 
\end{corollary}
\begin{proof}
By Corollary \ref{smallcase} we have that for $n \leq 3$ the matrices $\B{m}{r}$ do not depend on $\overline{m}$. Thus, they are all equal when $\overline{m}$ varies and we can only have one equivalence class. 
\end{proof}
In fact Eilers et al.\  \cite{seerapws:gccfg} established that $\widetilde \varphi(r) = 4$ if and only if $3\mid r$. As stated earlier, we shall see a general closed expression for $\widetilde \varphi$ in a later section.

\subsection{Equivalence of matrices}
To show equivalence of matrices we need to do some manipulations with matrices that might be a bit technical. So the following lemma simply establishes the equivalence of two matrices where every entry except for the diagonal is divisible by either $r$ or $\frac r2$ when $r$ is even.
\begin{lemma}\label{eqMatrices}
Let $r>2$ be given such that $r=2^ts$ for some $t\in \{0, 1\}$ and odd $s\in \N$.
 Suppose  that the two $n\times n$ upper triangular integer matrices $A, B$ have 1's in their diagonal, $r$ on the diagonal from $\langle 1, 2\rangle$ to $\langle n-1, n\rangle$, $\frac{r(r+1)}{2}$ on the diagonal from $\langle 1, 3\rangle$ to $\langle n-2, n\rangle$. Further, suppose $s$ divides every entry of $A-I$ and $B-I$. Then $A\sim B$.
\end{lemma}
\begin{proof}
We need to show that we can transform the matrix $A-I$ into $B-I$ by integer row and column operations. If $r$ is odd, every entry of $A-I$ and $B-I$ is divisible by $r$ by assumption, and the matrices are of the form
\begin{align*}
r \begin{pmatrix}
          0&1&   &  \\
          0&0&1& \\
          \vdots &\vdots&&\ddots \\
          0 & 0& \cdots &0 & 1\\
          0& 0& \cdots& 0&0
      \end{pmatrix}
\end{align*}
with integer entries in the upper right corner. All such matrices can easily be transformed into an upper triangular matrix with zeros everywhere except for the diagonal from $\langle 1, 2\rangle $ to $\langle n-1, n\rangle$ by row and column operations, so since $\sim$ is an equivalence relation, we get $A\sim B$.

Now, assume that $2\mid r$, but $ 4 \nmid r$ . Then the matrices $A-I$ and $B-I$ are of the form
\begin{align*}
A-I = \frac r2 \begin{pmatrix}
          0&2&r+1   &  \\
          0&0&2&r+1 \\
          \vdots &\vdots&&\ddots&\ddots \\
          0 & 0& \cdots &0 & 2&r+1\\
          0 & 0& \cdots &0 &0& 2\\
          0& 0& \cdots& 0&0&0
      \end{pmatrix}
\end{align*}
with integer entries in the upper right corner. We show that such a matrix can be transformed by row and column operations into the matrix
\begin{align*}
C - I := \frac r2 \begin{pmatrix}
          0&2&r+1   &0&\cdots & \cdots & 0  \\
          0&0&2&r+1 &0&\cdots & 0\\
          \vdots &\vdots&&\ddots&\ddots &\ddots & \vdots\\
          0 & 0& \cdots &0 & 2&r+1&0\\
          0 & 0& \cdots &0&0 & 2&r+1\\
          0 & 0& \cdots &0 &0&0& 2\\
          0& 0& \cdots& 0&0&0&0
      \end{pmatrix},
\end{align*}
which by transitivity shows $A\sim B$.

We proceed by induction on $n$. For $n=1, 2, 3$ all matrices on the described form will be identical and thus $\sim$-equivalent to $C$.

Now, assume that for $k<n$ every matrix on the described form are $\sim$-equivalent to $C$, and consider the $n\times n$ matrix $A$ on that form. By the induction hypothesis and considering $A$ as an $n-1\times n-1$ matrix with an added row and column, we can reduce $A$ by row and column operations to a matrix with diagonals like $C-I$ and zeroes everywhere else except for in the rightmost column. 

Using column operations we can now make every entry of that rightmost column (except for the $r-1$-entry) even without changing the rest of the matrix. If an entry of the column is odd, we can just subtract $r-1$ from it by subtracting the appropriate column.

Having made all those entries even they can all be eliminated by subtracting the 2 in the $(n-1)$st row an appropriate amount of times. And then the matrix $C-I$ is achieved, which concludes the proof.
\end{proof}
\noindent
Another useful result on when matrices are not $\sim$-equivalent is the following.
\begin{lemma} \label{notsim}
Let $A$ and $B$ be $n\times n$ upper rectangular matrices with $1$ in their diagonal. If every entry of $A-I$ and $B-I$ except for the entry $\langle 1, n\rangle$ is divisible by $k\in \N$ and $(A-I)\langle 1, n\rangle \not\equiv (B-I)\langle 1, n\rangle\pmod k$, then $A\not\sim B$.
\end{lemma}
\begin{proof}
Since every entry of $A-I$ and $B-I$ except the upper right is divisible by $k$, the upper right entry is invariant modulo $k$ under row and column operations. The conclusion follows.
\end{proof}

\section{The general case}
In general it is very difficult to find an explicit formula for $\B{r}{m}\langle i, j\rangle$ given arbitrary $r$ and $\overline m$. However, for the purpose of bounding $\varphi_r(n)$ from below and deciding $\widetilde \varphi(r)$ in the case where $4\nmid r$, it turns out to be sufficient to be able to compute $\B{r}{m}\langle i, j\rangle$ modulo $r$. 

Thus, this section sets out to develop techniques for assessing $\B{r}{m}\langle i, j\rangle$ modulo $r$. The main technical result is Theorem \ref{main} from which the exact value of $\widetilde \varphi(r)$ follows for $4\nmid r$ and a lower bound on $\varphi_r(n)$, which appears to be an equality when $4\nmid r$ (see Conjecture \ref{classes_conjecture}). Throughout the section, we will define $0^0=1$ and $0!=1$ for sake of simplicity.
\\

We start with the following lemma, which formally captures the technique which will be used multiple times in the proof of Theorem \ref{main}. 
\begin{lemma}\label{reduceSums}
	Let $D$ be a finite set, $p$ be a prime, $k,j\in \N$, $s, b\colon D\to \Z$ be functions, and $a:\N_0\times \N_0 \to \Z$ satisfies $\gcd \left( a(m, m), p \right)=1$ for all $m \leq j$. Define the function
	\begin{equation*}
		w(l) = \sum_{d\in D} s(d) \sum_{t=0}^l a(t, l)b(d)^t
	\end{equation*}
	and assume that $p^k\mid w(l)$ for $0\leq l < j$. Then
	\begin{equation*}
		w(j) \equiv \sum_{d\in D}a(j, j)s(d)b(d)^j \pmod{p^k}.
	\end{equation*}
\end{lemma}
\begin{proof}
First, we show by strong induction over $t$ that $p^k\mid \sum_{d\in D} s(d)b(d)^t$ for all $t<j$. For $t=0$ we have $p^k\mid \sum_{d\in D} s(d)a(0, 0)$ and since $\gcd(p^k, a(0,0))=1$ we get $p^k\mid \sum_{d\in D} s(d)$.
 Now, assume that $p^k\mid \sum_{d\in D} s(d)b(d)^t$ for all $t$ satisfying $0\leq t<m$ for some $m<j$. Then 
\begin{align*}
		0\equiv w(m) &= \sum_{d\in D} s(d) \sum_{t=0}^m a(t, m)b(d)^t\\
		             &\equiv \sum_{d\in D} a(m, m)s(d)b(d)^m \pmod{p^k}
	\end{align*}
	so $p^k\mid \sum_{d\in D} s(d)b(d)^m$ as $\gcd(a(m, m), p^k) = 1.$

Second, the fact that $p^k\mid \sum_{d\in D} s(d)b(d)^t$ for all $t<j$ yields
\begin{align*}
		w(j) &= \sum_{d\in D} s(d) \sum_{t=0}^j a(t, j)b(d)^t\\
			 &\equiv \sum_{d\in D}a(j, j)s(d)b(d)^j \pmod{p^k}.
	\end{align*}
\end{proof}
\noindent

\noindent Having proved the lemma we now turn to the main technical theorem of the section from which the remaining results follow naturally.

\begin{theorem}\label{main}
Let $p$ be an odd prime; $r>2$ and $n\leq p+1$ be given; and $\overline m=(m_1, \dots, m_n)$. Suppose that $p^k\mid r$ and $p^k\mid \B rs\ind 1a$ for every $\overline s\in (Z_r)^{p+1}$ and every $a<n$. Then 
\begin{align*}
	\B{r}{m}\langle 1, n\rangle \equiv \binom{r+n-2}{n-1} \prod_{k=2}^{n-1}m_k^{-1} \pmod{p^k}.
\end{align*}
\end{theorem}

\begin{proof}
For every vector $\overline m$ we can reduce the problem to considering a vector $\overline m'$ which satisfies $m_1'=m_2'=m_n'=1$ as follows. First, recall that no matter what vector we consider, we can always assume without loss of generality that its first and last entry is 1 since it does not affect any of the sides of the above expression by Lemma \ref{invariantFirstLast}. Second, as in the proof of Corollary \ref{change3} we can multiply $\overline m$ by a $m_2^{-1}$ to get $\overline m' =  m_2^{-1}\cdot \overline m$, which means that $m_2'=1$. Then the left hand side will not change since $\B rm\ind 1n = \B r{m'}\ind 1n$ by Lemma \ref{multiply} and the right hand side will satisfy
\begin{align*}
    \binom{r+n-2}{n-1} \prod_{k=2}^{n-1}m_k^{-1} \equiv \binom{r+n-2}{n-1} \prod_{k=2}^{n-1}bm_k^{-1} \pmod{p^k}
\end{align*}
since for $n<p+1$, $p^k\mid \binom{r+n-2}{n-1}$ and for $n=p+1$, $b^{n-2}\equiv 1\pmod p$ and $p^{k-1}\mid \binom{r+n-2}{n-1}$. Now, assuming $m_1'=m_n'=1$ yields the above. Thus, for the remaining proof we will assume that $m_1=m_2=m_n=1$.
\\

Now, let $\overline n_j$ denote the vector $(m_1, \dots, m_{n-j}, \underbrace{1, \dots, 1}_{j})$ and note that with the above assumption, $\overline n_1 = \overline m$ and $\overline n_{n-2}=\overline 1$.
\noindent Our approach will be to show that for all $1 \leq j < n-1$ we have 
\begin{equation}\label{goalOfHugeProof} 
	\B{r}{n_j} \langle 1, n\rangle \equiv m_{n-j}^{-1}g_j(m_1, \dots, m_{n-j-1})\pmod{p^k}
\end{equation}
for some integer function $g_j\colon (Z_r)^{n-j-1}\to \Z$ which is independent of $m_{n-j}$ and where $m_{n-j}^{-1}$ is the inverse of $m_{n-j}$ modulo $p^k$. Noting that \eqref{goalOfHugeProof} yields $\B{r}{n_{j+1}}\langle 1, n\rangle\equiv g(j, m_1, \dots, m_{n-j-1})\pmod {p^k}$, we get
\begin{align*}
	\B{r}{ n_{j}}\langle 1, n\rangle\equiv  \B{r}{n_{j+1}}\langle 1, n\rangle m_{n-j}^{-1} \pmod {p^k} ,
\end{align*}
 and applying this together with Theorem \ref{closedFormula} and $m_2=1$ gives us
\begin{align*} 
	\B{r}{m}\langle 1, n\rangle &= \B{r}{n_1}\langle 1, n\rangle\\
	&\equiv  \B{r}{n_2}\langle 1, n\rangle m_{n-1}^{-1} \\
	&\, \, \, \vdots\\
	&\equiv \B{r}{n_{n-2}} \langle 1, n\rangle \prod_{k=3}^{n-1}m_k^{-1}\\
	&= \B{r}{1} \langle 1, n\rangle\prod_{k=3}^{n-1}m_k^{-1}\\
	&=\binom{r+n-2}{n-1} \prod_{k=2}^{n-1}m_k^{-1} \pmod{p^k}.
	\end{align*}
	
Thus, all we need to do is prove that we can indeed write an expression for $\B{r}{n_j} \langle 1, n\rangle$ of the form \eqref{goalOfHugeProof}.

To do so, fix a $j$ with $1 \leq j < n-1$ and consider the graph $\NN{r}{n_j} $. We may write 
\begin{equation} \label{eqformel}
\B{r}{n_j} \langle 1, n \rangle = \sum_{q=0}^{r-1} L_{j}(q) S_j(q) 
\end{equation}
where $ S_j(q)$ denotes the number of paths on $\NN r{n_j}$ from $c_{1,0}$ to $c_{n-j, q}$ that are  subpaths of a legal path from $c_{1, 0}$ to $c_{n, 0}$ and that do not traverse any edges 
 in the $(n-j)$th subgraph, and similarly $L_j(q)$ is the number of paths on $\NN r{n_j}$ from $c_{n-j,q}$ to $c_{n,0}$  that are subpaths of a legal path from $c_{1, 0}$ to $c_{n, 0}$.
\\

We start our analysis by finding a formula for $ L_j(q)$. First, we consider the $(n-j+1)$th subgraph of $\NN r{n_j}$ and count the number of paths from $c_{n-j+1,i}$ to $c_{n,0}$ on $\NN r{n_j}$ that are subpaths of a legal path from $c_{1, 0}$ to $c_{n, 0}$ for each $0\leq i<r$. As in the proof of Theorem one can see choosing such a path as choosing a partition of $[r-i]_r$ into a sum of $j-1$ non-negative integers since $m_{n-j+1}=m_{n-j+2}=\dots=m_n=1$. Thus, the number of such paths is equal to $ \binom{[r-i]_r+j-1}{j-1}$.

Second, there are three cases to consider. When $i,q > 0$ there is exactly one path from $c_{n-j,q} $ to $c_{n-j+1,i}$ not traversing any edges in the $n-j+1$th subgraph that is a subpath of a legal path from $c_{1, 0}$ to $c_{n, 0}$ if and only if $ [m_{n-j}^{-1}q]_r \leq [m_{n-j}^{-1}i]_r $. Otherwise there are none. This is clear since such a path would be of the form
\begin{align*}
    c_{n-j, q}\to c_{n-j, [q+m_{n-j}]_r}\to \dots \to c_{n-j, i}\to c_{n-j+1, i}
\end{align*}
and zero is not a member of $\{q, q+m_{n-j}, \dots, i\}$ if and only if $ [m_{n-j}^{-1}q]_r \leq [m_{n-j}^{-1}i]_r $.
For $i=0$ there is exactly one such subpath for every $q$ and for $q= 0$ there is exactly one such subpath if and only if $i=0$.
Thus, for for $q > 0$ 
\begin{align*} 
L_j(q) &= \sum_{i=0}^{r-1} (1_{ \{ [m_{n-j}^{-1}q]_r \leq [m_{n-j}^{-1}i]_r \}}+ 1_{ \{i =0 \}})  \binom{[r-i]_r +j-1}{j-1} \\ =& \sum_{i=[m_{n-j}^{-1}q]_r}^{r} \binom{[r - i m_{n-j}]_r + j -1 }{j-1}
\end{align*}
Where $1_{\text{boolean}}$ is an  indicator function assuming the value 1 if is true and 0 otherwise, and where we changed $i =0$ terms into $i=r$ terms.
Introducing the new variable $\sigma = r-i$ we rewrite the sum as
\begin{equation}  \label{sumOfLk}
L_j(q) = \sum_{\sigma=0}^{[-m_{n-j}^{-1}q]_r} \binom{[\sigma m_{n-j}]_r + j -1 }{j-1} 
\end{equation}
Since evidently $L_j(0) = 1$, this formula holds even for $q = 0$ and thus for all $0\leq q < r$. \\

For $j=1$, \eqref{sumOfLk} yields $L_j(q) = [-m_{n-1}^{-1}q]_r$ and inserting in \eqref{eqformel} yields
\begin{align*}
    \B r{n_1} = \sum_{q=0}^{r-1}[-m_{n-1}^{-1}q]_rS_j(q) \equiv -m_{n-1}^{-1}\sum_{q=0}^{r-1} S_j(q)q \pmod{p^k}
\end{align*}
Since $S_j(q)q$ only depends on $m_1, m_2, \dots, m_{n-2}$, it follows that we can write $\B r{n_1}$ of the form \eqref{goalOfHugeProof}.

So let us consider the case when $j>1$. Inserting the expression \eqref{sumOfLk} into \eqref{eqformel} and substituting $d=r-q$ and noting that the $d=0$ is equal to the $d=r$ term yields
\begin{align*}
(j-1)! \B{r}{m} \langle 1, n \rangle  = \sum_{d=0}^{r-1} \sum_{\sigma=0}^{[m_{n-j}^{-1} d]_r} S_j([r-d]_r)\prod_{i=1}^{j-1}([\sigma m_{n-j}]_r + i).
\end{align*}
Expanding the product and introducing $s(d) = S_j([r-d]_r) $ we get the sum
\begin{equation*}\label{sum_formula}
(j-1)! \B{r}{n_j} \langle 1, n \rangle  = \sum_{d=0}^{r-1} \sum_{\sigma=0}^{[m_{n-j}^{-1}d]_r}  \sum_{t=0}^{j-1} a(t, j-1) ([\sigma m_{n-j}]_r)^t s(d)
\end{equation*}
for an integer function $a\colon \N_0\times \N_0 \to \Z$, where $a(j-1, j-1)=1$. Now, note that by the same reasoning we must also have for every $0\leq l < j-1$ that
\begin{align*}
w(l) & :=l!\B{r}{n_j} \langle 1, n-j+l+1 \rangle  \\
       &= \sum_{d=0}^{r-1} \sum_{\sigma=0}^{[m_{n-j}^{-1} d]_r} s(d) \sum_{t=0}^{l} a(t, l) ( [\sigma m_{n-j}]_r)^t
\end{align*}
where $a(l, l)=1$. By assumption,  $p^k$ divides $ \B r{n_j}\ind 1{n-j+l+1}$ for every $l<j-1$. Hence, $p^k\mid w(l)$ for $0\leq l<j-1$. Applying Lemma \ref{reduceSums} with the functions $s, a, b(\sigma):= [\sigma m_{n-j}]_r$, prime $p$, and exponent $t$, we get 
\begin{align*}
	(j-1)! \B{r}{n_j} \langle 1, n \rangle  & = w(j-1)\\
	&\equiv \sum_{d=0}^{r-1} \sum_{\sigma=0}^{[m_{n-j}^{-1} d]_r} a(j-1,j-1) s(d) ([\sigma m_{n-j}]_r)^{j-1}   \\
	&\equiv m_{n-j}^{j-1}\sum_{d=0}^{r-1} \sum_{\sigma=0}^{[m_{n-j}^{-1} d]_r } 
	 s(d)\sigma^{j-1} \pmod{p^k}. \numberthis \label{toInsertIn}
 \end{align*}
 By Faulhaber's formula \cite{jacobi}  in the convention $B_1 = \frac{1}{2}$, we can write
 \begin{equation*}
 	 \sum_{\sigma=0}^{[m_{n-j}^{-1} d]_r}  \sigma^{j-1} =  \frac{1}{j} \sum_{t=0}^{j-1} \binom{j}{j-t-1}B_{j-t-1}([m_{n-j}^{-1}d]_r)^{t+1},
 \end{equation*}
where $B_n$ is the $n$th Bernoulli number. Inserting in \eqref{toInsertIn}, noting that $p^k\mid r$ and multiplying both sides by $j$, we find that

\begin{align*}
	j! \B{r}{n_j} \langle 1, n \rangle  &\equiv  m_{n-j}^{j-2} \sum_{d=0}^{r-1}  s(d) d \sum_{t=0}^{j-1}   \binom{j}{j-t-1}B_{j-t-1} (m_{n-j}^{-1}d)^t  \pmod{p^k}
 \end{align*}
As it is a well-known fact that $ j! B_l, l<j,$ is an integer, we multiply by   $j! (m_{n-j}^{j-2})^{-1}$
 to ensure that each factor of each term is an integer 
\begin{align*}
	(m_{n-j}^{j-2})^{-1} j!^2 \B{r}{n_j}  \langle 1, n \rangle  &\equiv \sum_{d=0}^{r-1} s(d)d \sum_{t=0}^{j-1}  \binom{j}{j-t-1} j! B_{j-t-1} (m_{n-j}^{-1}d)^t  \pmod{p^k}.
 \end{align*}
To apply Lemma \ref{reduceSums} again we write 
\begin{align*}
    \tilde s(d)   & :=s(d)d, \\
    \tilde a(t,l) &:= \binom{l+1}{l-t} (l+1)! B_{l-t},\\
    \tilde b(d) &:= [m_{n-j}^{-1}d]_r,
\end{align*}
and considering the vectors $\overline{v_{l+1}} = (m_1, m_2, \dots, m_{n-j}, \underbrace{ 1, \dots, 1}_{l+1})$ one finds: 
\vspace{-0.3 cm}
\begin{align*}
    \tilde w(l) &:= \sum_{d=0}^{r-1} \tilde s(d) \sum_{t=0}^l \tilde a(t, l)\tilde b(d)^t \\ &\equiv (m_{n-j}^{l-1})^{-1} (l+1)!^2 \B r{v_{l+1}} \langle 1, n-j+l+1 \rangle\pmod {p^k}.\\ &\equiv (m_{n-j}^{l-1})^{-1} (l+1)!^2 \B r{n_j} \langle 1, n-j+l+1 \rangle\pmod {p^k}.
\end{align*}
 Now, $  \tilde a(l,l) =  \binom{l+1}{0} (l+1)! B_{0} = (l+1)!$ so for $ 0\leq l < j-1 < n-1\leq p$ we have $ \gcd( \tilde a(l,l), p)  = 1$. Further, by assumption $p^k\mid \B r{n_j}\ind 1{n-j+l+1}$ for $0\leq l<j-1$, so $p^k\mid \tilde w(l)$ for $0\leq l<j-1$. Thus, Lemma \ref{reduceSums} yields
 \begin{align*}
(m_{n-j}^{j-2})^{-1} j!^2 \B{r}{n_j} \langle 1, n \rangle  &=\tilde w(j-1)\\
	&\equiv  \sum_{q=0}^{r-1}s(d)d \tilde a(j-1,j-1) (m_{n-j}^{-1}d)^{j-1} \\ 
	& \equiv m_{n-j}^{1-j} \sum_{q=0}^{r-1} j! s(d) d^{j} \pmod{p^k} 
 \end{align*}
This means that 
  \begin{align*}
\B{r}{n_j} \langle 1, n \rangle  &\equiv  m_{n-j}^{-1} \sum_{q=0}^{r-1}  j!^{-1}s(d) d^{j} \pmod{p^k} ,
 \end{align*}
where we note that $j!^{-1}$ is well-defined because $j<n-1\leq p$ so $\gcd(j!, p)=1$. Since $s(d)$ only depends on $m_{1}, \dots, m_{n-j-1}$, it is clear that we can find $g_j$ satisfying \eqref{goalOfHugeProof} and we are done.
\end{proof}

\noindent Having proved the above theorem we can apply it to find $\widetilde \varphi(r) $ whenever $4 \nmid r$.  
We first use the theorem to prove the following lemma, which will give the first half of the proof.

\begin{lemma}\label{divides}
Let $r> 2$, $p$ be an odd prime, and $p^k\mid\mid r$ for some $k\in \N$. For every vector $\overline m$ with entries in $Z_r$ and every pair $a, b$ satisfying $0<b-a< p$ we have
$$
p^{k}\mid \B{r}{m} \langle a, b\rangle.
$$
\end{lemma}
\begin{proof}
We proceed by induction on the difference $n=b-a$. \\
When $n=1$ we have $ p^k\mid \B{r}{m} \langle a, b\rangle=r$. Now, suppose that $p^k\mid \B{r}{m}\langle a', b'\rangle$ for every $a', b'$ satisfying $0<b'-a'<n$ for some $n$ with $1<n< p$ and let $b-a=n$. Then we can apply Theorem \ref{main} with the indices $\langle 1, n\rangle$ shifted to $\langle a, b\rangle$ to get 
\begin{align*}
\B{r}{m}\langle a, b\rangle
&\equiv \binom{r-1+(b-a)}{b-a}\prod_{k=a+1}^{b-1}m_k^{-1}\\ 
&\equiv\frac{r \cdots  (r-1+(b-a))}{(b-a)!}\prod_{k=a+1}^{b-1}m_k^{-1}\\
&\equiv 0 \pmod{p^k},
\end{align*}
where the last equivalence follows since $p^k\mid \frac{r \cdots (r-1+(b-a))}{(b-a)!}$ because $b-a<p$ and $r$ divides the the numerator.
\end{proof}

\noindent Now, using the previous lemma and Theorem \ref{main} we obtain an upper bound on $\widetilde \varphi$ simply by pointing to two graphs that are not equivalent. In Theorem \ref{classesInequality} below we will establish a lower bound for the number of equivalence classes from which the result will follow. But for clarity we now give a short independent proof.

\begin{theorem}\label{phiUpperBound}
    Let $r>2$ be given and let $p$ be the smallest odd prime dividing $r$. Then $\widetilde\varphi(r)\leq p+1$.
\end{theorem}
\begin{proof}
    Let $k$ be such that $p^k\mid\mid r$, set 
$$\overline a=(\underbrace{1, \dots 1}_{p+1}) \text{ and } \overline b=(1, -1, \underbrace{1, \dots 1}_{p-1}),$$
and consider the matrices $A=\B{r}{a}, B=\B{r}{b}$. Then by Lemma \ref{divides} we have $p^k\mid A\langle a, b \rangle$ and $p^k\mid B\langle a, b \rangle $ for $a<b$ and $\ind ab\neq \ind 1{p+1}$. Using Theorem \ref{main} twice and noting that $(r+1)\cdots (r+p-1)\equiv (p-1)!\pmod {p^k}$, we get
\begin{align*}
A\langle 1, p+1\rangle &= \binom{r+p-1}{p} \prod_{k=2}^{p}a_k^{-1}\\
                       &= \frac{r}{p} \pmod{p^k},
\end{align*}
and
 \begin{align*}
B\langle 1, p+1\rangle &\equiv  \binom{r+p-1}{p} \prod_{k=2}^{p}b_k^{-1}\\
                       &\equiv -\frac{r}{p} \pmod{p^k},
\end{align*}
since $b_2=-1$.
It follows that $p^k$ divides every entry of $A-I$ and $B-I$ except for the entry $\langle 1, p+1\rangle$. Applying Lemma \ref{notsim} we get $\B ra\not\sim \B rb$ implying $\varphi_r(p+1)>1$ and the conclusion follows.
\end{proof}

\noindent Now, using the theorem we determine $\widetilde \varphi(r)$ whenever $4 \nmid r$.

\begin{theorem}\label{varphiisp1}
Let $r>2$ be given such that $4\nmid r$  and let $p$ be the smallest odd prime dividing $r$. Then $\widetilde \varphi(r) = p+1$. 
\end{theorem}
\begin{proof}
It follows from Lemma \ref{eqMatrices} and Lemma \ref{divides} that for every $n\leq p$ and every $\overline m \in (Z_r)^n$ we have $ \B{r}{m}\sim \B{r}{1}$, so $\varphi_r(n)=1$ for $n\leq p$. Thus, $\varphi(r)>p$. The conclusion now follows from Theorem \ref{phiUpperBound}. \end{proof}

The remaining part of this section deals with the number of equivalence classes, $\varphi_r(n)$.

\begin{notation}
Let $A=(a_{ij})$ be a matrix. Then we denote by $A[c,d]$ the partial square matrix
$$
\begin{pmatrix}
    a_{cc}& \cdots &a_{cd}\\
    \vdots & &\vdots\\
    a_{dc}&\cdots &a_{dd}
\end{pmatrix}.
$$
\end{notation}

\begin{lemma}\label{yessim}
Let $A, B$ be upper triangular matrices with $A\sim B$. Then $A[b,b+c]\sim B[b, b+c]$ for $b, c \in \N$ whenever the partial matrices are well-defined.
\end{lemma}
\begin{proof}
By the definition of $\sim$-equivalence, we have $A\sim B$ if and only if $A-I$ can be transformed into $B-I$ by pivots where a row can only be added to a row above it and a column can only be added to a column on its right. Noting that any such series of pivots on $A$ will act on the submatrix $(A-I)[b, b+c]$ as though they were simply pivots carried out on $(A-I)[b, b+c]$ as an independent matrix, it follows that $(A-I)[b, b+c] = A[b, b+c]-I$ can be transformed into $B[b, b+c]-I$ with pivots as described in our definition and the result follows.
\end{proof}

We introduce a necessary condition for two vectors $\overline m$ and $\overline n$ to have graphs with $\sim$-equivalent matrices.

\begin{theorem} \label{necessary}
Let $r>2$ have prime factorisation $r=2^j p_1^{\alpha_1}\cdots p_k^{\alpha_k}, j\in \N_0$ for distinct odd primes $p_i$. Further, let $\overline m, \overline m' \in (Z_r)^n$ be given such that $\B{r}{m}\sim \Bm{r}{m} $. Then for every $i$ with $1\leq i\leq k$ and every $t$ with $1\leq t\leq n-p_i$ we have
$$
\prod_{l=t+1}^{t+p_i-1}m_l \equiv \prod_{l=t+1}^{t+p_i-1}m'_l \pmod{p_i}.
$$
\end{theorem}
\begin{proof}
Assume for contradiction that for some $i, t$ we have 
$$
\prod_{l=t+1}^{t+p_i-1}m_l \not\equiv \prod_{l=t+1}^{t+p_i-1}m'_l \pmod{p_i}
$$
and consider the matrices $A=\B{r}{m}[t, t+p_i][t, t+p_i]$ and $B=\Bm{r}{m}[t, t+p_i][t, t+p_i]$. By Lemma \ref{yessim} we must have $A\sim B$ and by Lemma \ref{divides}, $p_i^{\alpha_i}$ divides every entry of $A-I$ and $B-I$ except the entry $\ind 1{p_i}$. For the entry $\ind 1{p_i}$ note that $p^{\alpha_i-1}\mid\mid \binom{r+p_i-1}{p_i}$ and that given integers $a, b, c$ such that $a\not\equiv b\pmod p$ and $p^{\alpha-1}\mid\mid c$ for a prime $p$, then $ac\not\equiv bc\pmod{p^{\alpha}}$. Combining these two observations yields  
\begin{align*}
    A\ind 1{p_i} &\equiv \binom{r+p_i-1}{p_i} \prod_{l=t+1}^{t+p_i-1}m_l^{-1}\\
    &\not\equiv \binom{r+p_i -1}{p_i} \prod_{l=t+1}^{t+p_i-1}m'^{-1}_l\\
    &\equiv B\ind 1{p_i} \pmod{p_i^{\alpha_i}}.
\end{align*}
Thus, by Lemma \ref{notsim} we have $A \not\sim B$. A contradiction.
\end{proof}

This necessary condition on $\sim$-equivalence translates directly into a lower bound on the number of equivalence classes, $\varphi_r(n)$.

\begin{theorem}\label{classesInequality}
Let $r>2$ have prime factorisation $r=2^j p_1^{\alpha_1}\cdots p_k^{\alpha_k}, j\in \N_0$ for odd distinct primes $p_i$. Then
\begin{equation*}
  \varphi_r(n) \geq \prod_{i=1}^k \ceil{(p_i-1)^{n-p_i}}.
\end{equation*}
\end{theorem}
\begin{proof}
For every $1\leq i \leq k$ we define a function $T_i\colon (Z_{p_i})^n\to (Z_{p_i})^{n-p_i}$ given by 
$$T_i(\overline m) = \left( \left[\prod_{l=t+1}^{t+p_i-1}m_l\right]_{p_i} \right)_{t=1}^{n-p_i}.$$ 
In case $n\leq p_i$, $T_i$ is simply the function $T_i\colon (Z_{p_i})^n\to \{1\}$.
To show that each $T_i$ is surjective, let $\overline m'\in (Z_{p_i})^{n-p_i}$ be a vector  and define $\overline m\in (Z_{p_i})^n$ as follows. 
\begin{align*}
    m_l = \begin{cases}
                1,& l<p_i\\
                m'_{l-p_i+1}\left[\prod_{q=l-p_i+2}^{l-1} m_{q}^{-1}\right]_{p_i} & l\geq p_i.
          \end{cases}
\end{align*}
Since the $t$'th entry of $T_i(\overline m)$ is given by 
\begin{align*}
    \left[\prod_{l=t+1}^{t+p_i-1}m_l\right]_{p_i}=\left[m_{t+p_i-1}\prod_{l=t+1}^{t+p_i-2}m_l \right]_{p_i} = m'_t
\end{align*}
it follows that $\overline m'\in T_i((Z_{p_i})^n)$ and thus, $T_i$ is surjective. 

Now, define the map $T\colon (Z_r)^n\to (Z_{p_1})^{n-p_1}\times \dots \times (Z_{p_k})^{n-p_k}$ by $T(\overline m) = (T_1(\overline m), T_2(\overline m), \dots, T_k(\overline m))$ in the natural way. Since each $T_i$ is surjective on $(Z_{p_i})^{n}\to (Z_{p_i})^{n-p_i}$ it follows by the Chinese Remainder Theorem that $T$ is also surjective. Now, for any two vectors $\overline m, \overline n\in (Z_r)^n$ such that $\B rm\sim\B rn$ we must have $T(\overline m)=T(\overline n)$ by Theorem \ref{necessary}. Thus, $T$ is an invariant of $\sim$-equivalence, it is surjective, and its codomain has $\prod_{i=1}^k \ceil{(p_i-1)^{n-p_i}}$ elements and it follows that indeed
\begin{equation*}
  \varphi_r(n) \geq \prod_{i=1}^k \ceil{(p_i-1)^{n-p_i}}.
\end{equation*}

\end{proof}

By the theorem we now have a lower bound on the number of equivalence classes, but we conjecture that the condition in Theorem \ref{necessary} is actually sufficient whenever $4 \nmid r$. This would then result in equality in the above Theorem \ref{classesInequality}, see Conjectures \ref{classes_conjecture} and \ref{necessary_conjecture}. Note further that using the inequality we can obtain Theorem \ref{phiUpperBound} and Theorem \ref{varphiisp1} since when $n = p+1$ where $p$ is the least odd prime dividing $r$ we will get at least $(p-1)$ classes.
\section{The case of multiples of four}
Until now we have not determined $\widetilde\varphi(r)$ in the special case where four divides $r$. This section will show that for $4\mid r$ we have $\widetilde \varphi (r)\leq 6$ with equality if and only if $3\nmid r$. To this end, we start with a few lemmas regarding specific entries of $\B rm$. Throughout the section we will change our notation slightly to make our calculations more natural, identifying the $r$th vertex of any subgraph of $\NN rm$ with the 0th.

\begin{lemma}\label{from1to4}
	Let $r>2$ be given with $2^t\mid r, t>1$ and let $\overline m \in (Z_r)^4$. Then 
	\begin{equation*}
		2^t \mid \B rm  \langle 1, 4\rangle.
	\end{equation*}
\end{lemma}
\begin{proof}
	By Corrollary \ref{change3}
	we can assume without loss of generality that $\overline m = (1, m_2, 1,1)$ for some $m_2\in Z_r$. We calculate $\B rm  \langle 1, 4\rangle$ by counting the number of legal paths from $c_{1, 0}$ to $c_{4, 0}$. We will sum over the last vertex $q, 1\leq q\leq r,$ of the second subgraph that each path visits. Denote by $S_2(q)$ the number of paths from $c_{1, 0}$ to $c_{2, q}$ that are subpaths of a legal path from $c_{1, 0}$ to $c_{4, 0}$ and similarly, let $L_2(q)$ denote the number of paths from $c_{2, q}$ to $c_{4, 0}$ that do not traverse any edges of the second subgraphs and that are subpaths of a legal path from $c_{1,0}$ to $c_{4, 0}$.Then 
	\begin{equation*}
		\B rm\langle 1, 4\rangle = \sum_{q=1}^r S_2(q)L_2(q).
	\end{equation*}
	\noindent
	First, it is not hard to see that $L_2(q) = r-q+1$ as $m_3 =1$. 
	\\
	Second, if we write $q=[tm_2]_r, 1\leq t\leq r$ we can see that for every subpath $\phi$ counted by $S_2(q)$ there must be a first vertex $c_{2, v}$ of the second subgraph that it visits. We must have $v\in \{[wm_2]_r\mid 1\leq w\leq t\}$ for else $\phi$ could never legally visit $c_{2, q}$. Further, there is exactly one subpath $\phi$ going through $c_{2, v}$ as specified, the path 
	$$c_{1, 0}\to c_{1, 1}\to \dots\to c_{1, v}\to c_{2, v}\to c_{2, [v+m_2]_r}\to c_{2, t}.$$
	It follows that 
	$$S_2(q)=\abs{\{[wm_2]_r\mid 1\leq w\leq t\}} = t\equiv qm_2^{-1}\pmod r,$$
	so we can calculate
	\begin{align*}
		\B rm \langle 1, 4\rangle \equiv \sum_{q=1}^r qm_2^{-1} (r-q+1)
		\equiv m_2^{-1}\sum_{q=1}^r q(r-q+1) \pmod r.
	\end{align*}
	By noting that $\B r1 \ind 14 \equiv \sum_{q=1}^r q(r-q+1) \pmod r$, it follows that
	\begin{align*}
	    \B rm \langle 1, 4\rangle \equiv m_2^{-1}\B r1\langle 1, 4\rangle \equiv m_2^{-1}\binom{r+2}{3} \equiv 0\pmod {2^{t}}
	\end{align*}
	by use of Theorem \ref{closedFormula}. 
\end{proof}

\begin{lemma}\label{from1to5}
	Let $r>2$ be given and assume that $2^t\mid\mid r $ for a $t>1$ and let $\overline m \in (Z_r)^5$. Then 
	\begin{equation*}
		2^{t-2}\mid \mid \B{r}{m} \langle 1, 5\rangle.
	\end{equation*}
\end{lemma}
\begin{proof}
	By Lemma  \ref{change3} we can assume without loss of generality that $\overline m = (1, m_2, 1, m_4, 1)$. We calculate $\B rm  \langle 1, 4\rangle$ by counting the number of legal paths from $c_{1, 0}$ to $c_{5, 0}$. We will sum over the last vertex $q, 1\leq q\leq r,$ of the second subgraph that each path visits. Denote by $S_2(q)$ be the number of paths from $c_{1, 0}$ to $c_{2, q}$ that are subpaths of a legal path from $c_{1, 0}$ to $c_{5, 0}$ and similarly, let $L_2(q)$ denote the number of paths from $c_{2, q}$ to $c_{5, 0}$ that do not traverse any edges of the second subgraph and are subpaths of a legal path from $c_{1,0}$ to $c_{5, 0}$.Then 
	\begin{equation*}
		\B rm\langle 1, 5\rangle = \sum_{q=1}^r S_2(q)L_2(q).
	\end{equation*}
	\noindent
	As in the proof of Lemma \ref{from1to4}, $S_2(q) \equiv qm_2^{-1}\pmod r$. It follows that
	\begin{align}\label{no2}
		\B rm \langle 1, 5\rangle &\equiv \sum_{q=1}^{r} qm_2^{-1}L_2(q) \equiv m_2^{-1}\BBB{r}{(1, 1, 1, m_4, 1)}\langle 1, 5\rangle \pmod{r}.
	\end{align}
	We proceed to calculate $\BBB r{(1, 1, 1, m_4, 1)} \langle 1, 5\rangle$ by almost the same approach as before. Write
	\begin{align*}
		\BBB r{(1, 1, 1, m_4, 1)}\langle 1, 5\rangle = \sum_{q=1}^{r}S_3(q)L_4(q),
	\end{align*}
	where $S_3(q)$ is the number of paths on $\NNN r{(1, 1, 1, m_4 ,1)}$  from $c_{1, 0}$ to $c_{3, q}$ that are subpaths of a legal path from $c_{1, 0}$ to $c_{5, 0}$. Further, $L_4(q)$ is the number of paths from $c_{3, q}$ to $c_{5, 0}$ that do not traverse any edge of the third subgraph and are subpaths of a legal path from $c_{1,0}$ to $c_{5, 0}$. Let $\phi$ be a path counted by $L_3(q)$ and let $c_{4, v}$ be the last vertex of the fourth subgraph that $\phi$ visits. By $L_3(q, v)$ we count the number of such $\phi$. Then
	\begin{align*}
	    \BBB r{(1, 1, 1, m_4, 1)}\langle 1, 5\rangle &= \sum_{q=1}^{r}S_3(q)\sum_{v=1}^rL_3(q, v)\\
	     &= \sum_{v=1}^r\sum_{q=1}^{r}S_3(q)L_3(q, v)\\
	     &= \sum_{q=1}^rS_3(q)L_3(q, r)+\sum_{v=1}^{r-1}\sum_{q=1}^rS_3(q)L_3(q, v).
	\end{align*}
	Since $\sum_{q=1}^rS_3(q)L_3(q, r)$ simply counts the number of legal paths from $c_{1, 0}$ to $c_{4, r}=c_{4, 0}$ that are the subpath of a legal path from $c_{1, 0}$ to $c_{5, 0}$, we have \begin{align*}
	    \sum_{q=1}^rS_3(q)L_3(q, r) = \BBB r{(1, 1, 1, m_4, 1)}\langle 1, 4\rangle = \binom{r+2}{3}\equiv 0\pmod {2^{t}}
	\end{align*}
	by Theorem \ref{closedFormula} and Lemma \ref{invariantFirstLast}. 
	Considering the case $1\leq v<r$ yields that $L_3(q, v)=0$ if $[qm_4^{-1}]_r>[vm_4^{-1}]_r$ since there is no legal path from $c_{4, q}$ to $c_{4, v}$ because such a path would visit $c_{4, 0}$ and $v\neq 0$. Further, if $[qm_4^{-1}]_r\leq [vm_4^{-1}]_r$ we have $L_3(q, v)=1$ since only the path 
	\begin{align*}
	    c_{3, q} \to c_{4, q}\to c_{4, q+m_4}\to \dots \to c_{4, v}\to c_{5, v}\to \dots \to c_{5, 0}
	\end{align*}
	satisfies the criteria. It follows that
	\begin{align*}
	    \BBB r{(1, 1, 1, m_4, 1)}\langle 1, 5\rangle &\equiv \sum_{v=1}^{r-1}\sum_{q=1}^rS_3(q)L_3(q, v)\\
	    &\equiv \sum_{q=1}^r \sum_{\substack{1\leq v<r \\ [qm_4^{-1}]_r\leq [vm_4^{-1}]_r}} S_3(q)\\
	    &= \sum_{q=1}^r [r-qm_4^{-1}]_r S_3(q) \pmod{2^t} 
	\end{align*}
	where the last equality follows since multiplying by $m_4^{-1}$ modulo $r$ induces a bijection on the set $\{1, \dots, r-1\}$, yielding 
	$$\abs{\{v\mid 1\leq v<r \wedge [qm_4^{-1}]_r\leq [vm_4^{-1}]_r\}}=[r-qm_4^{-1}]_r.$$ 
	So we get
	\begin{align*}
	    \BBB r{(1, 1, 1, m_4, 1)}\langle 1, 5\rangle &\equiv \sum_{q=1}^r [r-qm_4^{-1}]_r S_3(q)\\
	    &\equiv m_4^{-1} \sum_{q=1}^r -q S_3(q)\\
	    &\equiv m_4^{-1} \B r1\ind 15 \pmod {2^t}.
	\end{align*}
	Inserting in \eqref{no2} then finally yields
	\begin{align*}
	    \B rm \langle 1, 5\rangle &\equiv  m_2^{-1}\BBB{r}{(1, 1, 1, m_4, 1)}\langle 1, 5\rangle \\
	    &\equiv m_2^{-1}m_4^{-1} \B r1\ind 15 \\
	    &\equiv m_2^{-1}m_4^{-1} \binom{r+3}{4}\\
	    &\equiv s2^{t-2}\pmod{2^t}
	\end{align*}
	for some odd integer $s$ since $2^{t-2}\mid\mid \binom{r+3}4$ as $4\mid r$.
\end{proof}

\begin{lemma}\label{from1to6}
    Let $r>2$ be given. Then 
    \begin{align*}
        \BBB r{(1, 1, -1, 1, 1, 1)}\ind 16 = \tfrac{11}{20}r+\tfrac38r^2-\tfrac18r^3+\tfrac18r^4+\tfrac3{40}r^5
    \end{align*}
\end{lemma}
\begin{proof}
    In the graph $\NNN r{1, 1, -1, 1, 1, 1}$ let again $S_3(q)$ be the number of paths from vertex $c_{1, 0}$ to $c_{3, q}$ that are subpaths of a legal path from $c_{1, 0}$ to $c_{6, 0}$ such that $c_{3,q}$ is the last vertex visited in the third subgraph and let $L_3(q)$ be the number of paths from $c_{3, q}$ to $c_{6, 0}$ that does not traverse any edges of the third subgraph and are subpaths of a legal path from $c_{1, 0}$ to $c_{6, 0}$. We will find a closed form for each function.
    
	First, let $0<q<r$. Counting the paths of $S_3(q)$, we notice that there is exactly one	path from $c_{1, 0}$ to $c_{3, q}$ for every path from $c_{1, 0}$ to $c_{2,i}$ for $p>i\geq q$. This is the path
	\begin{align*}
	    c_{1, 0}\to\dots c_{2, i}\to c_{3, i}\to c_{3, i-1}\to \dots\to c_{3, q}.
	\end{align*}
	Since $m_1=m_2=1$ in this case, the number of paths from $c_{1, 0}$ to $c_{2, i}$ that are part of a legal path from $c_{1, 0}$ to $c_{6, 0}$ is $i$. Thus, 
	\begin{align*}
	    S_3(q) = \sum_{i=q}^{r-1} i = \frac{(r-q)(r+q-1)}{2}, 0<q<r.
	\end{align*}
	The function $L_3(q)$ is only counting paths that are traversing subgraphs with parameter $m_i=1$. We see by Corrollary \ref{smallcase} that 
	\begin{align*}
	    L_3(q) = \B{r-q+1}{1}\langle 1,3\rangle = \frac{(r-q+1)(r-q+2)}2.
	\end{align*}
	
	Second, for $q=0$ we have $S_3(0) = \frac{r(r+1)}{2}$ by Corollary \ref{smallcase} since this is simply $\BBB{r}{(1, 1, -1, 1, 1)}\langle 1, 3\rangle$.
	Further, there is only one legal subpath from $c_{4, 0}$ to $c_{6, 0}$ of a legal path from $c_{1, 0}$ to $c_{6, 0}$ so $L_3(0) = 1$.
	
	Thus, we have
	\begin{align*}
		\BBB r{(1, 1, -1, 1, 1)}\langle 1, 6\rangle &= \sum_{q=0}^{r-1} S_3(q)L_3(q)\\
		&= \frac{r(r+1)}{2} + \sum_{q=1}^{r-1} \frac{(r-q)(r-q+1)(r-q+2)(r+q-1)}{4}\\
		&= \tfrac{11}{20}r+\tfrac38r^2-\tfrac18r^3+\tfrac18r^4+\tfrac3{40}r^5,
	\end{align*}
	where the last equality follows by writing out the expression and applying Faulhaber's formula \cite{jacobi}.
\end{proof}

\begin{theorem}\label{varphiequals6}
	Let $r>2$ be given such that $4\mid r$. Then $\widetilde \varphi(r)= 6$ with equality if and only if $3\nmid r$.
\end{theorem}
\begin{proof}
    If $3\mid r$, we have $ \widetilde \varphi(r)\leq 4$ by Theorem \ref{phiUpperBound}, so we will now only consider the case when $3\nmid r$.
    
	First, we show that $\widetilde\varphi(r)>5$. Let $\overline m, \overline{m}'\in (Z_r)^5$ be given and let $X=\A{r}{m}$ and $Y=\Am{r}{m}$. We will demonstrate that $X\sim Y$ proving that $\varphi_r(5)=1$.
	Since $3\nmid r$ it follows from Lemma \ref{divides} that if $r=s2^t, 2\nmid s$ then $s$ will divide every entry of $\B rm$ and $\Bm r{m}$ except for the diagonal. Thus, by Lemmas \ref{from1to4} and \ref{from1to5} the matrices are of the following form.
	\begin{align*}
		X-I &= \begin{bmatrix}
			     0&r&\frac{r(r+1)}{2} &x_1 r& x_2 \frac r4\\
			     0&0&r&\frac{r(r+1)}{2}&x_3 r\\
			     0&0&0&r&\frac{r(r+1)}{2}\\
			     0&0&0&0&r\\
			     0&0&0&0&0\\
		     \end{bmatrix}\\
		  Y-I &= \begin{bmatrix}
			     0&r&\frac{r(r+1)}{2} &y_1 r& y_2 \frac r4\\
			     0&0&r&\frac{r(r+1)}{2}&y_3 r\\
			     0&0&0&r&\frac{r(r+1)}{2}\\
			     0&0&0&0&r\\
			     0&0&0&0&0\\
		     \end{bmatrix}
	\end{align*}
	for integers $x_1, x_2, x_3, y_1,y_2, y_3$, where $2\nmid x_2, y_2$. Now, reducing according to Definition \ref{maindefinition} in a number of steps, we get
	\begin{align*}
		X-I \stackrel{1}{\cong} \begin{bmatrix}
			     0&r&\frac{r}{2} &0&x_2 \frac r4\\
			     0&0&r&\frac{r(r+1)}{2}&0\\
			     0&0&0&r&\frac{r}{2}\\
			     0&0&0&0&r\\
			     0&0&0&0&0\\
		     \end{bmatrix} 
		     \stackrel{2}{\cong} \begin{bmatrix}
			     0&r&\frac{r}{2} &0&y_2 \frac r4\\
			     0&0&r&\frac{r(r+1)}{2}&0\\
			     0&0&0&r&\frac{r}{2}\\
			     0&0&0&0&r\\
			     0&0&0&0&0\\
		     \end{bmatrix}
		     \stackrel3\cong Y-I.
	\end{align*}
	Step 1 reduces the entries of the first row and last column of $X-I$ modulo $r$ by subtracting the fourth row and second column from the others. \\Step 2 adds the third column to the last column $\frac{y_2-x_2}2$ times and then subtracts the fourth row from the second  $\frac{y_2-x_2}2$ times. \\Step 3 is simply the reverse of step 1 except with $Y-I$ instead of $X-I$. It follows that $X\sim Y$.
	
	Second, we show that $\widetilde\varphi(r)\leq 6$, which completes the proof. Suppose that $5\mid r$. Then it follows by Theorem \ref{phiUpperBound} that $\widetilde\varphi(r)\leq 6$. So assume that $3, 5\nmid r$. Now, since $4\mid r$, Theorem \ref{closedFormula} yields 
	\begin{align*}
	    r\mid \B r1\ind 16 = \binom{r+4}5.
	\end{align*}
	Using Lemmas \ref{from1to4}, \ref{from1to5}, and \ref{from1to6} and noting that since $4 \mid r$ we have
	\begin{align*}
	    \tfrac{11}{20}r+\tfrac38r^2-\tfrac18r^3+\tfrac18r^4+\tfrac3{40}r^5\equiv \pm\frac r4\pmod r
	\end{align*}
	 we get by Lemma \ref{from1to6} that 
\begin{align*}
    \BBB r{(1, 1, -1, 1,1, 1)}-I &\stackrel1\cong  \begin{bmatrix}
			     0&r&\frac{r}{2} &0&x_1 \frac r4&\pm \frac r4\\
			     0&0&r&\frac{r(r+1)}{2}&x_2 r&x_3\frac r4\\
			     0&0&0&r&\frac{r(r+1)}{2}&0\\
			     0&0&0&0&r&\frac{r}{2}\\
			     0&0&0&0&0&r\\
			     0&0&0&0&0&0
		     \end{bmatrix}\\
		     &\stackrel{2}{\cong} \begin{bmatrix}
			     0&r&\frac{r}{2} &0&x_1 \frac r4&\frac r4\\
			     0&0&r&\frac{r}{2}&x_2 r&x_3\frac r4\\
			     0&0&0&r&\frac{r}{2}&0\\
			     0&0&0&0&r&\frac{r}{2}\\
			     0&0&0&0&0&r\\
			     0&0&0&0&0&0
		     \end{bmatrix}\\
		     &\stackrel{3}{\cong} \begin{bmatrix}
			     0&r&\frac{r}{2} &0&\frac r4&\frac r4\\
			     0&0&r&\frac{r}{2}&0&\frac r4\\
			     0&0&0&r&\frac{r}{2}&0\\
			     0&0&0&0&r&\frac{r}{2}\\
			     0&0&0&0&0&r\\
			     0&0&0&0&0&0
		     \end{bmatrix}.
\end{align*}
and
\begin{align*}
    \BBB r{(1, 1, 1, 1,1, 1)}-I &\stackrel1\cong  \begin{bmatrix}
			     0&r&\frac{r}{2} &0&y_1 \frac r4&0\\
			     0&0&r&\frac{r(r+1)}{2}&y_2 r&y_3\frac r4\\
			     0&0&0&r&\frac{r(r+1)}{2}&0\\
			     0&0&0&0&r&\frac{r}{2}\\
			     0&0&0&0&0&r\\
			     0&0&0&0&0&0
		     \end{bmatrix}\\
		     &\stackrel{2}{\cong} \begin{bmatrix}
			     0&r&\frac{r}{2} &0&y_1 \frac r4&0\\
			     0&0&r&\frac{r}{2}&y_2 r&y_3\frac r4\\
			     0&0&0&r&\frac{r}{2}&0\\
			     0&0&0&0&r&\frac{r}{2}\\
			     0&0&0&0&0&r\\
			     0&0&0&0&0&0
		     \end{bmatrix}\\
		     &\stackrel{3}{\cong} \begin{bmatrix}
			     0&r&\frac{r}{2} &0&\frac r4&0\\
			     0&0&r&\frac{r}{2}&0&\frac r4\\
			     0&0&0&r&\frac{r}{2}&0\\
			     0&0&0&0&r&\frac{r}{2}\\
			     0&0&0&0&0&r\\
			     0&0&0&0&0&0
		     \end{bmatrix}.
\end{align*}
for odd $x_1, x_2, x_3, y_1, y_2, y_3$ by the following steps. Step 1 reduces the first row and last column modulo $r$ by subtracting the second column and fifth row repeatedly from the other columns and rows. Step 2 subtracts the third column (fourth row) $\frac{r}2$ times from the fourth column (third row) and adds the second column (fifth row) $\frac r4$ times to the fourth column (third row). Step 3 reduces the entries $\ind 15$, $\ind 25$, and $\ind 26$ modulo $\frac r2$ by subtracting the fourth column and third row repeatedly from the fifth and sixth column and first and second row repeatedly. Note that the changes to entries $\ind 41$, $\ind 42$, $\ind 53$, and $\ind 63$ can be inverted by adding the second and third column to the fourth column and by adding the fourth and fifth row to the third row.

Now, dividing every entry by $\frac r4$, it follows that we have $\B r1\sim \mathsf{B}_{(r;(1,1,-1, 1,1,1)}$ if and only if
\begin{align*}
    \begin{bmatrix}
			     0&4&2 &0&1&\pm1\\
			     0&0&4&2&0&1\\
			     0&0&0&4&2&0\\
			     0&0&0&0&4&2\\
			     0&0&0&0&0&4\\
			     0&0&0&0&0&0
		     \end{bmatrix} &\cong \begin{bmatrix}
			     0&4&2 &0& 1&0\\
			     0&0&4&2&0&1\\
			     0&0&0&4&2&0\\
			     0&0&0&0&4&2\\
			     0&0&0&0&0&4\\
			     0&0&0&0&0&0
		     \end{bmatrix}.
\end{align*}
However, this can be checked to not be the case simply by solving the system of linear equations induced by Definition \ref{maindefinition} and finding that there are no solutions. Our conclusion follows. 
\end{proof}

\section{Concluding remarks}
Combining the results of the previous sections, we arrive at our main result, which answers the question of for which parameters $n$ and $r$ there only is a single, unique quantum lens space.
\begin{theorem}
\label{theorem_5_1}
Let $r>2$ and let $p$ be the smallest odd prime dividing $r$. Then
\begin{align*}
    \widetilde\varphi(r) = \begin{cases}
                           p+1, & 4\nmid r\\
                           \min\{6, p+1\}, & 4\mid r.
                       \end{cases}
\end{align*}
\end{theorem}
\begin{proof}
    For $4\nmid r$ this follows directly from Theorem \ref{varphiisp1}. Thus, let $4\mid r$. By Corollary \ref{3divides}, $\tilde\varphi(r)\geq 4$, and it follows that if $p=3$ we have $\tilde\varphi(r)=4$ by Theorem \ref{phiUpperBound} and if $p\neq 3$ we have $\tilde\varphi(r)=6$ by Theorem \ref{varphiequals6}.
\end{proof}

We recall that
$\widetilde\varphi(r)$ is the minimum $n$ for which there is an $m$ such that $C(L_q(r,\overline{1})\otimes K\not \simeq C(L_q(r, \overline{m})\otimes K$ so that our result explains exactly how to find the smallest dimension where the $m$-vector influences the stable isomorphism class of the quantum lens space for any fixed $r$. In fact, using Proposition 14.5 in \cite{seerapwsny:gccfg} we get that  
$\widetilde\varphi(r)$ is the minimum $n$ for which there is an $m$ such that $ C(L_q(r,\overline{1})\not \simeq C(L_q(r,\overline{m})$.

Further, for the case when the quantum lens space is not uniquely given, we studied the number of equivalence classes arising by varying the parameter $\overline m\in (Z_r)^n$. A lower bound on the number of such equivalence classes, Theorem \ref{necessary}, was found by giving a necessary condition for two quantum lens spaces to be isomorphic, Theorem \ref{necessary}. However, computer experiments suggest that this necessary condition is in fact even sufficient when $4\nmid r$.  We thus conjecture the following which we have confirmed by computer experiments for $r\in\{3, 5, 6, 9\}$ and $n\leq 8$ and for $r\in \{10, 15, 21\}$ and $n\leq 7$.

\begin{conjecture} \label{necessary_conjecture}
Let $r=2^t\cdot p_1^{\alpha_1}\cdots p_k^{\alpha_k}, t\in \{0, 1\}$. Further, let $\overline m, \overline m' \in (Z_r)^n$ be given. Then $\B{r}{m}\sim \Bm{r}{m}$ if and only if for every $1\leq i\leq k$ and $1\leq t\leq n-p_i$ we have
$$
\prod_{l=t+1}^{t+p_i-1}m_l \equiv \prod_{l=t+1}^{t+p_i-1}m'_l \pmod{p_i}.
$$
\end{conjecture}
\noindent
This conjecture is true if and only if we have equality in Theorem \ref{classesInequality} when $4\nmid r$, so an equivalent conjecture is the following. 
\begin{conjecture}\label{classes_conjecture}
Let $r>2$ have the prime factorization $r=2^t\cdot p_1^{\alpha_1}\cdots p_k^{\alpha_k}, t\in \{0, 1\}$. Then
\begin{equation*}
  \varphi_r(n) = \prod_{i=1}^k \ceil{(p_i-1)^{n-p_i}}.
\end{equation*}
\end{conjecture} 
Proving these conjectures seems hard to achieve using the methods of this paper, however, since determining equivalence of matrices once they become sufficiently large is a complex task unless one can find better invariants to rely on. Worth noting is that proving Conjectures \ref{necessary_conjecture} and \ref{classes_conjecture} would yield the following satisfactory result, which resounds well with the overall findings of this paper. 
\begin{conjecture}
The equivalence classes of $S_{r, n}/\sim$ all have the same number of members.
\end{conjecture}

\paragraph{Acknowledgement} 
The authors wish to thank prof.\ Søren Eilers for his guidance throughout this entire process. From posing the problem to providing final edits, his help has been invaluable in presenting these results.

\bibliographystyle{siam}
\bibliography{XM-article.bib}

\end{document}